\documentclass[11pt,reqno,letterpaper]{amsart}
\usepackage[a4paper,textwidth=17.25cm,textheight=22.50cm]{geometry}

\usepackage[english]{babel}
\usepackage[latin1]{inputenc}
\usepackage[T1]{fontenc} 
\usepackage{latexsym}
\usepackage{amsthm}
\usepackage{epic}
\usepackage{amsfonts}
\allowdisplaybreaks
\usepackage{rotating}
\usepackage[T1]{fontenc}
\usepackage{amsmath} 
\usepackage{amssymb} 
\usepackage[all]{xy}  
\usepackage{graphicx}  
\usepackage{xcolor}
\usepackage{mathrsfs}
\usepackage{euscript}
\usepackage{listings}
\usepackage{mathtools}

\usepackage{hyperref}
\hypersetup{
    colorlinks,
    linkcolor=red,
    citecolor=blue,
    urlcolor=green,
    breaklinks=true
}

\makeatletter
\newcommand*{\rom}[1]{\expandafter\@slowromancap\romannumeral #1@}
\makeatother

\theoremstyle{plain}
\newtheorem{thm}{Theorem}[section]
\newtheorem*{thm*}{Main Theorem}
\newtheorem{prop}[thm]{Proposition}

\newtheorem{cor}[thm]{Corollary}
\newtheorem{lmm}[thm]{Lemma}

\theoremstyle{definition}
\newtheorem{dfn}[thm]{Definition}

\newtheorem{rmk}[thm]{Remark}
\newtheorem{ex}[thm]{Example}

\makeatletter
      \def\@setcopyright{}
      \def\serieslogo@{}
      \makeatother

\begin{document}

\author{Luca Sodomaco}
\address{Dipartimento di Matematica e Informatica ``Ulisse Dini'', University of Florence, Italy}
\email{luca.sodomaco@unifi.it}
\subjclass[2000]{14M20, 15A18, 15A69, 15A72, 65H17}

\title{\textbf{The product of the eigenvalues of a symmetric tensor}}

\begin{abstract}
\noindent We study E-eigenvalues of a symmetric tensor $f$ of degree $d$ on a finite-dimensional Euclidean vector space $V$, and their relation with the E-characteristic polynomial of $f$. We show that the leading coefficient of the E-characteristic polynomial of $f$, when it has maximum degree, is the $(d-2)$-th power (respectively the $((d-2)/2)$-th power) when $d$ is odd (respectively when $d$ is even) of the $\widetilde{Q}$-discriminant, where $\widetilde{Q}$ is the $d$-th Veronese embedding of the isotropic quadric $Q\subseteq\mathbb{P}(V)$. This fact, together with a known formula for the constant term of the E-characteristic polynomial of $f$, leads to a closed formula for the product of the E-eigenvalues of $f$, which generalizes the fact that the determinant of a symmetric matrix is equal to the product of its eigenvalues.
\end{abstract}

\maketitle

\section{Introduction}

Let $(V,\langle\cdot\hspace{0.6mm},\cdot\rangle)$ be a real $(n+1)$-dimensional Euclidean space and denote with $\|\cdot\|$ the norm induced by $\langle\cdot\hspace{0.6mm},\cdot\rangle$. Our object of study is the vector space $\operatorname{Sym}^dV$ of degree $d$ symmetric tensors on $V$. An excellent reference for the algebraic geometry for spaces of tensors is \cite{landsberg2011tensors}. Any element $f\in\operatorname{Sym}^dV$ can be treated in coordinates as an element of $\mathbb{R}[x_1,\ldots,x_{n+1}]_d$, namely a degree~$d$ homogeneous polynomial in the indeterminates $x_1,\ldots,x_{n+1}$. The projective hypersurface defined by the vanishing of $f$ is denoted by $[f]$.

The notions of \emph{E-eigenvalue} and \emph{E-eigenvector} of a symmetric tensor were proposed independently by Lek-Heng Lim and Liqun Qi in \cite{lim2005singular,qi2005eigenvalues} in the more general setting of $(n+1)$-dimensional tensors of order~$d$, namely elements of $V^{\otimes d}$. There are different types of eigenvectors and eigenvalues in the literature, see \cite{cartwright2013number,hu2013determinants,ni2007degree,qi2007eigenvalues,qi2017tensor}. Although the notions of E-eigenvalues and E-eigenvectors of tensors arise mainly in the context of approximation of tensors, which deals usually with real tensors, for our investigations we need to extend the Euclidean space $(V,\langle\cdot\hspace{0.6mm},\cdot\rangle)$ to its complexification $(V^{\mathbb{C}},\langle\cdot\hspace{0.6mm},\cdot\rangle_{\mathbb{C}})$. In the following, we will use the same notation in $V$ and $V^{\mathbb{C}}$ for the bilinear form $\langle\cdot\hspace{0.6mm},\cdot\rangle$ and the norm $\|\cdot\|$.

\begin{dfn}\label{def eigenvalue eigenvector}
Given $f\in\operatorname{Sym}^dV$, a non-zero vector $x\in V^{\mathbb{C}}$ such that $\|x\|=1$ is called an \emph{E-eigenvector} of $f$ (where the ``E'' stands for ``Euclidean'') if there exists $\lambda\in\mathbb{C}$ such that $x$ is a solution of the equation
\begin{equation}\label{new interpretation eigenvectors}
\frac{1}{d}\nabla f(x)=\lambda x.
\end{equation}
The scalar $\lambda$ corresponding to $x$ is called an \emph{E-eigenvalue} of $f$, while the pair $(\lambda,x)$ is called an \emph{E-eigenpair} of $f$. The corresponding power $x^d\in\operatorname{Sym}^dV^{\mathbb{C}}$ is called an \emph{E-eigentensor} of $f$. In particular, for even order~$d$, $(\lambda,x)$ is an E-eigenpair of $f$ if and only if $(\lambda,-x)$ is so; for odd order~$d$, $(\lambda,x)$ is an E-eigenpair of $f$ if and only if $(-\lambda,-x)$ is so. If $x\in V^{\mathbb{C}}$ is a solution of (\ref{new interpretation eigenvectors}) such that $\|x\|=0$, we call~$x$ an \emph{isotropic eigenvector} of $f$.
\end{dfn}

The factor $\frac{1}{d}$ appearing in (\ref{new interpretation eigenvectors}) follows the notation in \cite{qi2005eigenvalues} conformed to the symmetric case. Observe that, if $(\lambda,x)$ satisfies (\ref{new interpretation eigenvectors}), then $(\alpha^{d-2}\lambda,\alpha x)$ satisfies (\ref{new interpretation eigenvectors}) for any non-zero $\alpha\in\mathbb{C}$. This is why we impose the additional quadratic equation $\|x\|=1$ in Definition \ref{def eigenvalue eigenvector}. When $d=2$, the definition of E-eigenvalue and E-eigenvector is not the same as the standard definition of eigenvalue and eigenvector of a symmetric matrix, as a non-zero complex vector $x$ satisfying $\|x\|=0$ is excluded in the definition of E-eigenvector.

In this article we investigate a fundamental tool for computing the E-eigenvalues of a symmetric tensor, namely its \emph{E-characteristic polynomial}. We recall its definition (for the definition of the resultant of $m$~homogeneous polynomials in $m$ variables see Section \ref{section: preliminaries}).

\begin{dfn}\label{characteristic polynomial}
Given $f\in\operatorname{Sym}^dV$, when $d$ is even the \emph{E-characteristic polynomial} $\psi_f$ of $f$ is defined by $\psi_f(\lambda)\coloneqq\operatorname{Res}(F_{\lambda})$, where $\lambda\in\mathbb{C}$ and $\operatorname{Res}(F_{\lambda})$ is the resultant of the $(n+1)$-dimensional vector
\begin{equation}\label{def charpoly even}
F_{\lambda}(x)\coloneqq
\frac{1}{d}\nabla f(x)-\lambda \|x\|^{d-2}x.
\end{equation}
When $d$ is odd, the E-characteristic polynomial is defined as $\psi_f(\lambda)\coloneqq\operatorname{Res}(G_{\lambda})$, where $\operatorname{Res}(G_{\lambda})$ is the resultant of the $(n+2)$-dimensional vector
\begin{equation}\label{def charpoly odd}
G_{\lambda}(x_0,x)\coloneqq
\begin{pmatrix}
x_0^2-\|x\|^2\\
\frac{1}{d}\nabla f(x)-\lambda x_0^{d-2}x
\end{pmatrix}.
\end{equation}
\end{dfn}
For $d=2$, the E-characteristic polynomial agrees with the characteristic polynomial of a symmetric matrix $A$, namely $\psi_A(\lambda)=\det(A-\lambda I)$. In this case, the roots of $\psi_A$ are all the eigenvalues of $A$, and if the entries of $A$ are real, then the roots of $\psi_A$ are all real by the Spectral Theorem. Moreover, the leading coefficient of $\psi_A$ is 1, implying that its constant term is equal to the product of the eigenvalues of $A$, that is the determinant of $A$.

The interesting fact is that this happens only for $d=2$: as we will see throughout the paper, given $f\in\operatorname{Sym}^dV$ with $d>2$, then some of the roots of the E-characteristic polynomial $\psi_f$ may not be real even though the coefficients of $f$ are real. However, there exist symmetric tensors with only real E-eigenvalues, as shown by Maccioni in \cite{maccioni1972number} and Kozhasov in \cite{kozhasov2017fully}. Moreover, the leading coefficient of $\psi_f$ is a homogeneous polynomial over $\mathbb{Z}$ in the coefficients of $f$ with positive degree for $d>2$. 

Our main result describes the product of the E-eigenvalues of a symmetric tensor $f$ when $\psi_f$ has maximum degree. In the following, $\widetilde{Q}$ denotes the Veronese embedding of the \emph{isotropic quadric} $Q\coloneqq\{\|x\|^2=x_1^2+\cdots+x_{n+1}^2=0\}\subseteq\mathbb{P}^n$, whereas $\Delta_{\widetilde{Q}}(f)$ is the $\widetilde{Q}$-\emph{discriminant} of $f$. For the definition of the polynomial $\Delta_{\widetilde{Q}}(f)$ see Section \ref{section: preliminaries}. Moreover, for all $n\geq 1$ we define the integer $N\coloneqq n+1$ for $d=2$, whereas $N\coloneqq((d-1)^{n+1}-1)/(d-2)$ for $d\geq 3$.
\begin{thm*}\label{the theorem}
Consider a real symmetric tensor $f\in\operatorname{Sym}^dV$ for $d\geq 2$. If $f$ admits the maximum number $N$ of E-eigenvalues (counted with multiplicity), then their product is
\begin{equation}\label{product eigenvalues of the theorem}
\lambda_1\cdots\lambda_N=\pm\frac{\operatorname{Res}\left(\frac{1}{d}\nabla f\right)}{\Delta_{\widetilde{Q}}(f)^{\frac{d-2}{2}}}.
\end{equation}
\end{thm*}
For the proof of the Main Theorem see Section \ref{section: leading coefficient}. We note (see Lemma \ref{lemma 7.2 cinesi}) that the assumption of the Main Theorem is satisfied for a general $f$, and it corresponds geometrically to the fact that the hypersurface $[f]$ is transversal to $Q$ (see Remark \ref{rmk: on transversality}). The formula (\ref{product eigenvalues of the theorem}) generalizes to the class of symmetric tensors the known fact that the determinant of a symmetric matrix is the product of its eigenvalues. In particular, we underline that the polynomial $\operatorname{Res}\left(\frac{1}{d}\nabla f\right)$ appearing in the numerator of (\ref{product eigenvalues of the theorem}) is equal to the classical \emph{discriminant} $\Delta_d(f)$ of $f$ times a constant factor (for the definition of discriminant of a homogeneous polynomial and a relation between the polynomials $\operatorname{Res}\left(\frac{1}{d}\nabla f\right)$ and $\Delta_d(f)$ see \cite[Proposition \rom{13}, 1.7]{gelfand2008discriminants}).

Among all the preliminary facts needed for the proof of the Main Theorem, we want to stress two of them in particular. First of all, the E-eigenvalues of $f\in\operatorname{Sym}^dV$ are roots of $\psi_f$, but the converse is true only for \emph{regular} symmetric tensors (see Definition \ref{irregular tensor} and \cite[Theorem 4]{qi2007eigenvalues}). The second fact is due to Cartwright and Sturmfels (see \cite[Theorem 5.5]{cartwright2013number}).

\begin{thm}[Cartwright-Sturmfels]\label{expected number of E-eigenvalues}
Every symmetric tensor $f\in\operatorname{Sym}^dV$ has at most $N$ distinct E-eigenvalues when $d$ is even, and at most $N$ pairs $(\lambda,-\lambda)$ of distinct E-eigenvalues when $d$ is odd. This bound is attained for general symmetric tensors.
\end{thm}

This fact was previously conjectured in \cite{ni2007degree}. In \cite{oeding2013eigenvectors}, Oeding and Ottaviani review Cartwright and Sturmfels' formula and propose an alternative geometric proof based on Chern classes, with various generalizations. However, this result had already essentially been known in complex dynamics due to Forn\ae ss and Sibony, who in \cite{fornaess1994complex} discuss global questions of iteration of rational maps in higher dimension.

These two results combined together show that the degree of the E-characteristic polynomial $\psi_f$ is equal to $N$ (or $2N$, depending on $d$ even or odd), whereas it is smaller than the ``expected'' one exactly when $f$ admits at least an isotropic eigenvector. This in particular motivated our research on the geometric meaning of the vanishing of the leading coefficient of $\psi_f$. Therefore the Main Theorem describes that, if the coefficients of $f$ annihilate the polynomial $\Delta_{\widetilde{Q}}(f)$, then some of the E-eigenvalues of $f$ have gone ``to infinity'': in practice, $f$ admits at least an isotropic eigenvector whose corresponding eigenvalue does not appear as a root of the E-characteristic polynomial~$\psi_f$. We stress that both numerator and denominator in (\ref{product eigenvalues of the theorem}) are orthogonal invariants of $f$, namely polynomials in the coefficients of $f$ that are invariant under the orthonormal linear changes of coordinates in $f$. The product of the E-eigenvalues of $f$ is \emph{a priori} equal to the right-hand side of (\ref{product eigenvalues of the theorem}) times a constant factor depending only on $n$ and $d$. Using the definitions of resultant and $\widetilde{Q}$-\emph{discriminant}, we prove that this constant factor is (in absolute value) 1 by specializing to the family of \emph{scaled Fermat polynomials}. However, the identity (\ref{product eigenvalues of the theorem}) is given up to sign since the definition of E-eigenvalue has this sign ambiguity.

This paper is organized as follows. After setting the notation, in Section \ref{section: preliminaries} we give some basic notions on resultants and on the dual of a hypersurface. In Section \ref{section: characteristic polynomial}, we recall the properties of the E-eigenvectors and the isotropic eigenvectors of a symmetric tensor $f$, and the first known properties of its E-characteristic polynomial $\psi_f$. In particular we point out that the coefficients of $\psi_f$ are orthogonal invariants of~$f$, and recall the fact proved in \cite{li2012characteristic} that the constant term of $\psi_f$ is equal (up to a constant factor) to the resultant of $\frac{1}{d}\nabla f$, for $d$ even, or equal to the square of the resultant of $\frac{1}{d}\nabla f$, for $d$ odd. Section \ref{section: leading coefficient} is devoted to the proof of the Main Theorem, before restating some useful facts from \cite{holme1988geometric,li2012characteristic}. Finally, in the first part of Section \ref{section: examples} we focus on the case of binary forms and rephrase some remarkable results in \cite{li2012characteristic}, whereas in the second part we stress with a concrete example how the presence of an isotropic eigenvector of $f\in\operatorname{Sym}^dV$ affects the geometry of the hypersurface $[f]$.

\section{Preliminaries}\label{section: preliminaries}

Consider the a real $(n+1)$-dimensional Euclidean space $(V,\langle\cdot\hspace{0.6mm},\cdot\rangle)$, where $\langle\cdot\hspace{0.6mm},\cdot\rangle:V\times V\to\mathbb{R}$ is a positive definite symmetric bilinear form on $V$. The quadratic form associated to $\langle\cdot\hspace{0.6mm},\cdot\rangle$ is $q:V\to\mathbb{R}$ defined by $q(x)\coloneqq\langle x,x\rangle=\|x\|^2$. The set of automorphisms $A\in\operatorname{Aut}(V)$ that preserve $\langle\cdot\hspace{0.6mm},\cdot\rangle$, i.e., such that $\langle Ax,Ay\rangle=\langle x,y\rangle$ for all $x,y\in V$, forms the orthogonal group $\operatorname{O}(V)$ and is a subgroup of $\operatorname{Aut}(V)$. The special orthogonal group $\operatorname{SO}(V)$ is defined as the set of all $A$ in $\operatorname{O}(V)$ with determinant 1. An $\operatorname{SO}(V)$-invariant (or orthogonal invariant) for $f\in\operatorname{Sym}^dV$ is a polynomial in the coefficients of~$f$ that does not vary under the action of $\operatorname{SO}(V)$ on the coefficients of~$f$, where the above-mentioned action is the one induced by the linear action of $\operatorname{SO}(V)$ on the coordinates of $V$. If we fix a basis on $V$, we identify $V$ with $\mathbb{R}^{n+1}$ and we consider the bilinear form defined by $\langle x,y\rangle=x_1y_1+\cdots+x_{n+1}y_{n+1}$ for all $x=(x_1,\ldots,x_{n+1})$, $y=(y_1,\ldots,y_{n+1})\in\mathbb{R}^{n+1}$: in this case, the associated quadratic form is $q(x)=\|x\|^2=x_1^2+\cdots+x_{n+1}^2$ for all $x=(x_1,\ldots,x_{n+1})\in\mathbb{R}^{n+1}$.

The bilinear symmetric form $\langle\cdot\hspace{0.6mm},\cdot\rangle$ can be extended to a bilinear symmetric form on $\operatorname{Sym}^dV$. Given $x,y\in V$ and the corresponding $d$-th powers $x^d$, $y^d$, we set $\langle x^d,y^d\rangle\coloneqq\langle x, y\rangle^d$. By linearity this defines~$\langle\cdot\hspace{0.6mm},\cdot\rangle$ on the whole $\operatorname{Sym}^dV$. In particular, given $f,g\in\operatorname{Sym}^dV$, we define the \emph{norm} of $f$ by $\|f\|\coloneqq\sqrt{\langle f,f\rangle}$ and the distance function between $f$ and $g$ by $d(f,g)\coloneqq\|f-g\|=\sqrt{\langle f-g,f-g\rangle}$.

We denote by $v_{n,d}\colon\mathbb{P}(V)\to\mathbb{P}(\operatorname{Sym}^dV)$ the Veronese map, which sends $[x]\in\mathbb{P}(V)$ to $[x^d]\in\mathbb{P}(\operatorname{Sym}^dV)$. The image of $v_{n,d}$ is denoted by $V_{n,d}$ and is the projectivization of the subset of $(\operatorname{Sym}^dV)^{\vee}$ consisting of $d$-th powers of linear polynomials or, in other words, rank one symmetric tensors on $V$. By definition, a \emph{critical rank one symmetric tensor} for $f\in\operatorname{Sym}^dV$ is a critical point of the distance function from $f$ to the affine cone over $V_{n,d}$.

The following proposition is well known and defines the notion of \emph{resultant} of a set of $m$ homogeneous polynomials in $m$ variables (see \cite{cox2006using,gelfand2008discriminants}).

\begin{prop}\label{def resultant}
Let $f_1,\ldots,f_m$ be $m$ homogeneous polynomials of positive degrees $d_1,\ldots,d_m$ respectively in the variables $x_1,\ldots,x_m$. Then there is a unique polynomial $\operatorname{Res}(f_1,\ldots,f_m)$ over $\mathbb{Z}$ in the coefficients of $f_1,\ldots,f_m$ such that
\begin{itemize}
\item[$i)$] $\operatorname{Res}(f_1,\ldots,f_m)=0$ if and only if the system $f_1=\cdots=f_m=0$ has a solution in $\mathbb{P}^{m-1}$.
\item[$ii)$] $\operatorname{Res}\left(\frac{1}{d}\nabla f\right)=(a_1\cdots a_m)^{(d-1)^{m-1}}$, where $f(x_1,\ldots,x_m)=a_1x_1^{d}+\cdots+a_mx_m^{d}$, $a_1,\ldots,a_m\in\mathbb{C}$, is the scaled Fermat polynomial.
\item[$iii)$] $\operatorname{Res}(f_1,\ldots,f_m)$ is irreducible, even when regarded as a polynomial over $\mathbb{C}$ in the coefficients of $f_1,\ldots,f_m$.
\end{itemize}
\end{prop}

The normalization assumption of $ii)$ coincides with the classical definition made in \cite[Theorem \rom{3}, 2.3 and Theorem \rom{3}, 3.5]{cox2006using} and in \cite[p. 427]{gelfand2008discriminants}.

The degree of the resultant is known in general.

\begin{prop}\label{degree resultant m polynomials}
$\operatorname{Res}(f_1,\ldots,f_m)$ is a homogeneous polynomial of degree $d_1\cdots d_{i-1}d_{i+1}\cdots d_m$ with respect to the coefficients of $f_i$ for all $i=1,\ldots,m$. Hence the total degree of $\operatorname{Res}(f_1,\ldots,f_m)$ is
\[
\deg\operatorname{Res}(f_1,\ldots,f_m)=\sum_{i=1}^m d_1\cdots d_{i-1}d_{i+1}\cdots d_m.
\]
In particular, when all the forms $f_1,\ldots,f_m$ have the same degree $d$, the resultant has degree $d^{m-1}$ in the coefficients of each $f_i$, namely $\deg\operatorname{Res}(f_1,\ldots,f_m)=md^{m-1}$.
\end{prop}

%Among the properties of the resultant, we will apply the fact that, when all the forms $f_1,\ldots,f_m$ have the same degree $d$, it behaves well when $f_1,\ldots,f_m$ are replaced by linear combinations of them.

%\begin{prop}\label{prop: resultant under linear trasformations}
%If $f_1,\ldots,f_m$ have the same degree $d$ and $g_i=\sum_{j=1}^ma_{ij}f_j$ for all $i=1,\ldots,m$, where $(a_{ij})$ is an invertible matrix with entries in $\mathbb{C}$, then
%\[
%\operatorname{Res}(g_1,\ldots,g_m)=\det(a_{ij})^{d^{m-1}}\operatorname{Res}(f_1,\ldots,f_m).
%\]
%\end{prop}

%\begin{ex}\label{ex: resultant fermat}
%Proposition \ref{prop: resultant under linear trasformations} is useful for the following example. Consider the scaled Fermat polynomial $f(x_1$,$\ldots$,$x_m)$ $=$ $a_1x_1^{d}+\cdots+a_mx_m^{d}$, $a_1,\ldots,a_m\in\mathbb{C}$. Then
%\[
%\operatorname{Res}\left(\frac{1}{d}\nabla f\right)=\operatorname{Res}\left(a_1x_1^{d-1},\ldots,a_{n+1}x_{n+1}^{d-1}\right)=(a_1\cdots a_m)^{(d-1)^{m-1}}.
%\]
%\end{ex}

The notion of resultant is closely related to the classical notion of \emph{discriminant} of a homogeneous polynomial of degree $d$ in $m$ variables, as pointed out in \cite{gelfand2008discriminants}. The problem of computing the discriminant of a homogeneous polynomial is a particular case of a more general geometric problem, that is, finding the equations of the dual of a variety (see \cite{gelfand2008discriminants,holme1988geometric,tevelev2007projectively}).

\begin{dfn}\label{dual variety}
Let $X\subseteq\mathbb{P}^n$ be an irreducible projective variety and denote by $X_{sm}$ its smooth locus. The \emph{dual variety} of $X$ is
\[
X^{\vee}\coloneqq\overline{\left\{H\in(\mathbb{P}^n)^{\vee}\ |\ T_PX\subseteq H\mbox{ for some }P\in X_{sm}\right\}},
\]
where the closure is taken with respect to the Zariski topology.
\end{dfn}

\begin{dfn}\label{conormal variety}
Let $X\subseteq\mathbb{P}^n$ be an irreducible projective variety. The \emph{conormal variety} of $X$ is
\[
Z(X)\coloneqq\overline{\{(P,H)\in\mathbb{P}^n\times(\mathbb{P}^n)^{\vee}\ |\ P\in X_{sm}\mbox{ and }T_PX\subseteq H\}}.
\]
\end{dfn}
Consider the projections $\pi_1\colon Z(X)\to X_{sm}$ and $\pi_2\colon Z(X)\to(\mathbb{P}^n)^{\vee}$ of $Z(X)$. In particular $\pi_1$, $\pi_2$ are the restrictions of the canonical projections $pr_1\colon\mathbb{P}^n\times(\mathbb{P}^n)^{\vee}\to\mathbb{P}^n$, $pr_2\colon\mathbb{P}^n\times(\mathbb{P}^n)^{\vee}\to(\mathbb{P}^n)^{\vee}$. Note that, by Definition \ref{dual variety}, $X^{\vee}$ coincides with the image of $\pi_2$ and is an irreducible variety. Moreover, since $\dim(Z(X))=n-1$, it follows that $\dim(X^{\vee})\leq n-1$ and we expect that in ``typical'' cases $X^{\vee}$ is a hypersurface.

\begin{dfn}
Let $X\subseteq\mathbb{P}^n$ be a projective variety. If $X^{\vee}$ is a hypersurface, then it is defined by the vanishing of a homogeneous polynomial, denoted by $\Delta_X$ and called the $X$-\emph{discriminant}. We assume the $X$-discriminant to have relatively prime integer coefficients: in this way, $\Delta_X$ is defined up to sign. If $X^{\vee}$ is not a hypersurface, then we set $\Delta_X\coloneqq 1$.
\end{dfn}

If $X\subseteq\mathbb{P}^n$ is an irreducible variety such that $X^{\vee}$ is a hypersurface, then $\Delta_X$ is an irreducible homogeneous polynomial over the complex numbers. When $X$ is the Veronese variety $V_{n,d}$, then it is known that, for all $d>1$, $X^{\vee}$ is a hypersurface and its equation, the $V_{n,d}$-discriminant, coincides up to a constant factor with the discriminant $\Delta_d(h)$ of a homogeneous polynomial $h$ of degree $d$ in $n+1$ variables.

\section{The E-characteristic polynomial of a symmetric tensor}\label{section: characteristic polynomial}

In this section we recall the main properties of E-eigenvectors and isotropic eigenvectors of a symmetric tensor $f$. After this, we treat more in detail the properties of the E-characteristic polynomial of $f$. 

Consider again the Definition \ref{def eigenvalue eigenvector}. The first consequence is the following property.
\begin{prop}\label{prop: identity lambda}
Let $f\in\operatorname{Sym}^dV$. If $(\lambda,x)$ is an E-eigenpair of $f$, then $\lambda=f(x)$.
\end{prop}
\begin{proof}
Apply the operator $\langle\cdot\hspace{0.6mm},x\rangle$ on both sides of equation (\ref{new interpretation eigenvectors}). Then we have
\[
\left\langle\frac{1}{d}\nabla f(x),x\right\rangle=\langle \lambda x,x\rangle.
\]
Using Euler's identity, the left hand side of last identity is equal to $f(x)$, whereas by linearity and the fact that $x$ has norm 1 the right-hand side is equal to $\lambda$.
\end{proof}

A remarkable fact observed in \cite{lim2005singular,qi2005eigenvalues} is that the E-eigenvectors of $f\in\operatorname{Sym}^dV$ correspond to the critical points of the function $f(x)$ restricted on the unit sphere $S^n\coloneqq\{x\in\mathbb{R}^{n+1}\ |\ \|x\|=1\}$. Hence the E-eigenvectors of $f$ are the normalized solutions $x$, in orthonormal coordinates, of:
\[
\mbox{rank}
\begin{pmatrix}
\nabla f(x)\\
x
\end{pmatrix}
\leq 1.
\]

Moreover, we recall an alternative interpretation of the eigenvectors of a symmetric form, meaning that an \emph{eigenvector} of $f$ is any solution of equation (\ref{new interpretation eigenvectors}), whether it has unit norm or not.

\begin{thm}[Lim, variational principle]
Given $f\in\operatorname{Sym}^dV$, the critical rank one symmetric tensors for $f$ are exactly of the form $x^d$, where $x$ is an eigenvector of $f$.
\end{thm}

This interpretation is used by Draisma, Ottaviani and Tocino in \cite{draisma2017best}, where they deal more in general with the \emph{best rank $k$ approximation problem} for tensors.

Looking at Definition \ref{def eigenvalue eigenvector}, a natural question is whether E-eigenvalues could change under an orthonormal linear change of coordinates in $V$ (see \cite[Theorem 1]{qi2007eigenvalues} and \cite[Theorem 2.20]{qi2017tensor}).

\begin{thm}\label{orthogonal invariance E-eigenvalues}
Given $f\in\operatorname{Sym}^dV$, the set of the E-eigenvalues of $f$ is a $\operatorname{SO}(V)$-invariant of $f$.
\end{thm}

In particular Theorem \ref{orthogonal invariance E-eigenvalues} states that the symmetric functions of the E-eigenvalues of $f$ are orthogonal invariants of $f$, giving rise to the following corollary (see \cite[Theorem 3.3]{li2012characteristic}):
\begin{cor}
Given $f\in\operatorname{Sym}^dV$, all the coefficients of the E-characteristic polynomial $\psi_f$ are $\operatorname{SO}(V)$-invariants of $f$. 
\end{cor}
Given $f\in\operatorname{Sym}^dV$, we observe that for $d$ even there exists a non-zero constant $c\in\mathbb{Z}$ such that
\begin{equation}\label{characteristic polynomial discr even case}
\psi_f(\lambda)\coloneqq\operatorname{Res}(F_{\lambda}(x))=c\cdot\Delta_d\left(f(x)-\lambda\|x\|^d\right),
\end{equation}
where the $(n+1)$-dimensional vector $F_{\lambda}(x)$ has been introduced in (\ref{def charpoly even}) (see again \cite[Proposition \rom{13}, 1.7]{gelfand2008discriminants}). On the other hand, a relation equivalent to (\ref{characteristic polynomial discr even case}) is no longer possible for $d$ odd: in (\ref{def charpoly odd}), an additional variable $x_0$ is required to make the polynomial $\psi_f$ well defined.

In the study of the E-characteristic polynomial $\psi_f$, a crucial role is played by a family of particular symmetric tensors, the ones admitting at least a singular point on the isotropic quadric $Q$.
\begin{dfn}\label{irregular tensor}
A symmetric tensor $f\in\operatorname{Sym}^dV$ is \emph{irregular} if there exists a non-zero vector $x\in V^{\mathbb{C}}$ such that $\|x\|=0$ and $\nabla f(x)=0$. Otherwise $f$ is called \emph{regular}.
\end{dfn}

The first fact on irregular symmetric tensors is that, when $d>2$, their E-characteristic polynomial is identically zero.

\begin{prop}\label{zero charpoly}
Given $f\in\operatorname{Sym}^dV$ with $d>2$, if $f$ is irregular then $\psi_f$ is the zero polynomial.
\end{prop}
\begin{proof}
Suppose that $f$ is irregular. Then, by Definition \ref{irregular tensor} there exists a non-zero vector $x\in V$ such that $\|x\|=0$ and $\nabla f(x)=0$. Looking at Definition \ref{characteristic polynomial}, this implies that, for $d>2$ even, $x$ is a solution of the system $F_{\lambda}(x)=0$ for all $\lambda\in\mathbb{C}$, whereas for $d>2$ odd $(0,x)$ is a solution of the system $G_{\lambda}(x_0,x)=0$ for all $\lambda\in\mathbb{C}$. By resultant theory, this means that $\psi_f(\lambda)=0$ for all $\lambda\in\mathbb{C}$, namely $\psi_f$ is identically zero.
\end{proof}

\begin{rmk}
The statement of Proposition \ref{zero charpoly} is no longer true for $d=2$. In fact, for $d=2$ and any $n\geq 1$ there exist irregular symmetric tensors $f\in\operatorname{Sym}^d\mathbb{C}^{n+1}$ such that $\psi_f$ is not identically zero. For example, the polynomial $f(x)=(x_1+\sqrt{-1}x_2)^2+x_3^2+\cdots+x_{n+1}^2$ is irregular because the vector $(1,\sqrt{-1},0,\ldots,0)$ is a solution of $\nabla f(x)=0$, whereas one can easily check that $\psi_f(\lambda)=\lambda^2(1-\lambda)^{n-1}$, hence it is not identically zero.
\end{rmk}

The notion of regularity of a symmetric tensor plays a crucial role in the following result.

\begin{thm}\label{eigenvalues as roots of charpoly}
Suppose that $d\geq 3$. Given $f\in\operatorname{Sym}^dV$, every E-eigenvalue of $f$ is a root of the E-characteristic polynomial $\psi_f$. If $f$ is regular, then every root of $\psi_f$ is an E-eigenvalue of $f$.
\end{thm}
\begin{proof}
For completeness we recover and adapt the proofs in \cite[Theorem 4]{qi2007eigenvalues} and in \cite[Theorem 2.23]{qi2017tensor}.
Suppose that $x\in V^{\mathbb{C}}$ is an E-eigenvector of $f$ and $\lambda\in\mathbb{C}$ is the E-eigenvalue associated with $\lambda$. Then looking at Definition \ref{characteristic polynomial}, when $d$ is even we get that $x$ and $-x$ are non-zero solutions of the system $F_{\lambda}(x)=0$; when $d$ is odd, $(1,x)$ and $(-1,-x)$ are non-zero solutions of the system $G_{\lambda}(x_0,x)=0$. Therefore $\lambda$ is a root of $\psi_f$ by Proposition \ref{def resultant}.

On the other hand, suppose that $f$ is regular and let $\lambda\in\mathbb{C}$ be a root of $\psi_f$. By Definition \ref{characteristic polynomial} and Proposition \ref{def resultant}, when $d$ is even there exists a non-zero vector $x\in V$ such that $F_{\lambda}(x)=0$ for that $\lambda$; when $d$ is odd, there exists a non-zero vector $x\in V$ and $x_0\in\mathbb{C}$ such that $G_{\lambda}(x_0,x)=0$ for that $\lambda$. If $\|x\|=0$, both $F_{\lambda}(x)=0$ and $G_{\lambda}(x_0,x)=0$ yield the condition $\nabla f(x)=0$, which cannot be satisfied because of the regularity of $f$. Hence $\|x\|\neq 0$ and we consider $\tilde{x}=x/\|x\|$. Therefore, when $d$ is even the equation (\ref{new interpretation eigenvectors}) is satisfied by $(\lambda,\tilde{x})$ and $(\lambda,-\tilde{x})$, while for $d$ odd it is satisfied by $(\lambda,\tilde{x})$ and $(-\lambda,-\tilde{x})$. This implies that $\lambda$ is an E-eigenvalue of $f$.\qedhere
\end{proof}
\begin{ex}
Let us consider the case in which $d$ is even and $f=\|x\|^d$. Then equation (\ref{new interpretation eigenvectors}) becomes $\|x\|^{d-2}x=\lambda x$: this means, if $d=2$, that every non-zero vector $x\in V$ such that $\|x\|=1$ is an E-eigenvector of $f$ with E-eigenvalue $\lambda=1$ (and in fact the E-characteristic polynomial of $f$ is $\psi_f(\lambda)=(\lambda-1)^{n+1}$). Instead for $d>2$ every non-zero vector $x\in V$ such that $\|x\|=1$ is an E-eigenvector of $f$ with corresponding E-eigenvalue $\lambda=1$, and every non-zero vector $x\in V$ such that $\|x\|=0$ is an isotropic eigenvector of $f$. In particular $f$ is irregular for $d>2$, and in fact in this case the E-characteristic polynomial of $f$ is identically zero by Proposition \ref{zero charpoly}.
\end{ex}

The greatest difference among eigenvectors of a symmetric matrix and eigenvectors of a symmetric tensor of degree $d>2$ is related to the presence or not of isotropic eigenvectors. Suppose that $f\in\operatorname{Sym}^dV$ admits an isotropic eigenvector $x$ and let $P\coloneqq[x]$ be the corresponding point of the isotropic quadric $Q$. In the same fashion of Proposition \ref{prop: identity lambda}, this time we have that $f(x)=0$, that is, $P\in[f]$. As we will see in Section \ref{section: leading coefficient}, equation (\ref{new interpretation eigenvectors}) acquires a new interesting meaning: the isotropic eigenvectors of $f$ are all the non-zero vectors $x$ such that $[x]\eqqcolon P\in [f]\cap Q$ and $P$ is singular for $[f]$ (and hence $f$ is irregular) or $P$ is smooth for $[f]$ and $[f]$ is tangent to $Q$ at $P$.

We study more in detail the coefficients of the E-characteristic polynomial of a symmetric tensor. Given a general $f\in\operatorname{Sym}^dV$, from Theorem \ref{expected number of E-eigenvalues} and Theorem \ref{eigenvalues as roots of charpoly} we have that $\deg(\psi_f)\leq N$ for $d$ even, where $N\coloneqq n+1$ for $d=2$, whereas $N\coloneqq((d-1)^{n+1}-1)/(d-2)$ for $d\geq 3$. Thus $\psi_f$ can be written as
\begin{equation}\label{rmk on writing of charpoly}
\psi_f(\lambda)=\sum_{j=0}^{N}c_j\lambda^j,
\end{equation}
where for all $j=0,\ldots,N$ the coefficient $c_j=c_j(n,d)$ is a homogeneous polynomial in the coefficents of $f$. Otherwise if $d$ is odd and $(\lambda,x)$ is an E-eigenpair of $f$, then $(-\lambda,-x)$ is an E-eigenpair of $f$ as well. This means that for $d$ odd the E-characteristic polynomial $\psi_f$ has maximum degree $N$ in $\lambda^2$ and in particular it contains only even power terms of $\lambda$. Hence $\psi_f$ can be written explicitly as
\begin{equation}\label{rmk on writing of charpoly, odd case}
\psi_f(\lambda)=\sum_{j=0}^{N}c_{2j}\lambda^{2j}.
\end{equation}

Now we focus on the constant term of the E-characteristic polynomial $\psi_f$. In particular we recover the fact that, when non-zero, the constant term of $\psi_f$ is a power of $\operatorname{Res}\left(\frac{1}{d}\nabla f\right)$ times a constant factor, as in the proof of \cite[Theorem 3.5]{li2012characteristic}.

\begin{thm}\label{thm: constant term of charpoly}
Let $f\in\operatorname{Sym}^dV$. Then for even $d$ we have that 
\begin{equation}\label{c_0(d) d even}
c_0=c\cdot\operatorname{Res}\left(\frac{1}{d}\nabla f\right),
\end{equation}
while for odd $d$ we have that
\begin{equation}\label{c_0(d) d odd}
c_0=c\cdot\operatorname{Res}\left(\frac{1}{d}\nabla f\right)^2
\end{equation}
for some constant $c\in\mathbb{Z}$ depending on $n$ and $d$.
\end{thm}

\begin{proof}
The relations (\ref{c_0(d) d even}) and (\ref{c_0(d) d odd}) are trivially satisfied when $f$ is irregular (compare with Proposition \ref{zero charpoly}), so we can assume $f$ regular. When $d$ is even, from relation (\ref{def charpoly even}) we have that
\[
c_0=\psi_f(0)=\left.\operatorname{Res}(F_{\lambda})\right|_{\{\lambda=0\}}=c\cdot\operatorname{Res}(F_0)=c\cdot\operatorname{Res}\left(\frac{1}{d}\nabla f\right)
\]
for some constant $c=c(n,d)\in\mathbb{Z}$.

Now suppose that $d$ is odd. From relation (\ref{def charpoly odd}) we have that
\[
c_0=\psi_f(0)=\left.\operatorname{Res}(G_{\lambda})\right|_{\{\lambda=0\}}=c\cdot\operatorname{Res}(G_{0}),\quad
G_{0}(x_0,x)=
\begin{pmatrix}
x_0^2-\|x\|^2 \\
\frac{1}{d}\nabla f(x)
\end{pmatrix}
\]
for some constant $c=c(n,d)\in\mathbb{Z}$. In order to prove relation (\ref{c_0(d) d odd}), it is sufficient to prove that 
\begin{equation}\label{equivalent relation}
\operatorname{Res}(G_0)=\operatorname{Res}\left(\frac{1}{d}\nabla f\right)^2.
\end{equation}
First of all, we prove that the system
\begin{equation}\label{equation constant term proof}
\left\{
\begin{array}{l}
x_0^2-\|x\|^2=0 \\
\frac{1}{d}\nabla f(x) = 0 
\end{array}
\right.
\end{equation}
has a nonzero solution if and only if $\operatorname{Res}\left(\frac{1}{d}\nabla f\right)=0$. Let $(x_0,x)$ be a non-zero solution of (\ref{equation constant term proof}). In particular, $x$ is a non-zero solution of $\nabla f(x)=0$. Thus, $\operatorname{Res}\left(\frac{1}{d}\nabla f\right)=0$. On the other hand, suppose that $\operatorname{Res}\left(\frac{1}{d}\nabla f\right)=0$. Then $\nabla f(x)=0$ admits a non-zero solution $x$ and $(\|x\|,x)$ is a non-zero solution of (\ref{equation constant term proof}).

Hence the equations $\operatorname{Res}(G_0)=0$ and $\operatorname{Res}\left(\frac{1}{d}\nabla f\right)=0$ define the same variety. By definition $\operatorname{Res}\left(\frac{1}{d}\nabla f\right)$ is an irreducible polynomial over $\mathbb{Z}$ in the coefficients of $f$. Therefore
\[
\operatorname{Res}(G_0)=\operatorname{Res}\left(\frac{1}{d}\nabla f\right)^k
\]
for some positive integer $k$. Since the polynomial $x_0^2-\|x\|^2$ is quadratic, from Proposition \ref{degree resultant m polynomials} we have that $\operatorname{Res}(G_0)$ is a homogeneous polynomial in the coefficients of $\partial f/\partial x_1,\ldots,\partial f/\partial x_{n+1}$ of degree $2(d-1)^n$. On the other hand, the degree of $\operatorname{Res}\left(\frac{1}{d}\nabla f\right)^k$ is $k(d-1)^n$. Therefore relation (\ref{equation constant term proof}) is satisfied only if $k=2$. This completes the proof.
\end{proof}

We apply the following result when we study the degree of the leading coefficient and of the constant term of $\psi_f$, viewed as polynomials in the coefficients of $f$ (see \cite[Proposition 3.6]{li2012characteristic}).
\begin{prop}\label{degrees of coefficients of charpoly}
Consider $f\in\operatorname{Sym}^dV$ and its E-characteristic polynomial $\psi_f$ written as in (\ref{rmk on writing of charpoly}), (\ref{rmk on writing of charpoly, odd case}).
\begin{itemize}
\item[$i)$] When $d$ is even, $c_i$ is a homogeneous polynomial in the coefficients of $f$ with degree $(n+1)(d-1)^n-i$. In particular $\deg(c_{N})=(n+1)(d-1)^n-N\eqqcolon\varphi_n(d)$, where the integer $N$ has been introduced in the Main Theorem. In particular $\varphi_n(2)=0$ for all $n\geq 1$.
\item[$ii)$] When $d$ is odd, $c_{2i}$ is a homogeneous polynomial in the entries of $f$ with degree $2(n+1)(d-1)^n-2i$. In particular $\deg(c_{2N})=2(n+1)(d-1)^n-2N=2\varphi_n(d)$.
\end{itemize}
\end{prop}

%\begin{proof}
%When $d$ is even, considering the polynomial $F_{\lambda}$ as in (\ref{def charpoly even}), we may see that $F_{\lambda}(x)=0$ is a system of homogeneous polynomials in $x_1,\ldots,x_{n+1}$. Every equation of this system has the same degree $d-1$. Thus, from Proposition \ref{degree resultant m polynomials}, $\psi_f(\lambda)$ is a homogeneous polynomial in the entries of $f$ and $\lambda$ of degree $(n+1)(d-1)^n$. Hence $c_i(n,d)$ is a homogeneous polynomial in the entries of $f$ with degree $(n+1)(d-1)^n-i$. In particular the leading coefficient $c_{h(n,d)}(n,d)$ has degree $\deg(c_{h(n,d)}(n,d))=\varphi_n(d)$.

%When $d$ is odd, considering $G_{\lambda}(x_0,x)$ as in (\ref{def charpoly odd}), then $G_{\lambda}(x_0,x)=0$ is a system of homogeneous polynomials in $x_0,x_1,\dots,x_{n+1}$. The coeffcients of the first equation of this system are either $1$ or $-1$, while the last two equations have the same degree $d-1$. Thus, again from \ref{degree resultant m polynomials}, $\psi_f(\lambda)$ is a homogeneous polynomial in the entries of $f$ and $\lambda$ of degree $2(n+1)(d-1)^n$. Hence $c_{2i}(n,d)$ is a homogeneous polynomial in the entries of $f$ with degree $2(n+1)(d-1)^n-2i$. In particular the leading coefficient $c_{2h(n,d)}(n,d)$ has degree $\deg(c_{2h(n,d)}(n,d))=2\varphi_n(d)$. 
%\end{proof}

\begin{rmk}\label{rmk: on the polynomial phi(n,d)}
It can be easily showed that the polynomial $\varphi_n(d)$ defined in Proposition \ref{degrees of coefficients of charpoly} is a strictly increasing function in the variable $d$. This fact, together with Proposition \ref{degrees of coefficients of charpoly}, implies that $c_{N}$ (respectively $c_{2N}$) has positive degree in the coefficients of $f$ for all $n\geq 1$ and $d>2$.
\end{rmk}

We have this natural question: is there a geometric meaning for the vanishing of the leading coefficient of $\psi_f$? The answer is positive and will be stated in Proposition \ref{lemma 7.1 cinesi}.

\section{Proof of the Main Theorem}\label{section: leading coefficient}

In this section we give the proof of the Main Theorem. The proof starts with an example: in fact, the next lemma studies the product of the E-eigenvalues of a particular class of symmetric tensors, the scaled Fermat polynomials $f(x_1,\ldots,x_{n+1})=a_1x_1^d+\cdots+a_{n+1}x_{n+1}^d$, where $a_1,\ldots,a_{n+1}\in\mathbb{C}$. This result is important to prove the identity (\ref{product eigenvalues of the theorem}) up to sign in the statement of the Main Theorem.

\begin{lmm}\label{scaled Fermat polynomials}
Let $d\geq 2$ and consider the scaled Fermat polynomial $f=a_1x_1^d+\cdots+a_{n+1}x_{n+1}^d$, where $a_1,\ldots,a_{n+1}\in\mathbb{C}$. The product $\lambda_1\cdots\lambda_N$ of the E-eigenvalues of $f$, where $N$ is the number defined in Theorem \ref{expected number of E-eigenvalues}, can be written as
\begin{equation}\label{product_eigenvalues}
\lambda_1\cdots\lambda_N=\frac{\operatorname{Res}\left(\frac{1}{d}\nabla f\right)}{h^{\frac{d-2}{2}}},
\end{equation}
where $h=h(a_1,\ldots,a_{n+1})$ is a homogeneous polynomial of degree $2\varphi_n(d)/(d-2)$ and the polynomial $\varphi_{n,d}$ has been defined in Proposition \ref{degrees of coefficients of charpoly}. Moreover, the leading term of $h$ with respect to the lexicographic term order is monic and it is equal to
\[
LT_{Lex}(h)=\prod_{s=1}^{n+1}a_{s}^{2\frac{(d-1)^n-(d-1)^{s-1}}{d-2}}.
\]
\end{lmm}

\begin{proof}
In this case, rewriting the number $N$ of E-eigenvalues as $N=\sum_{j=1}^{n+1}{n+1\choose j}(d-2)^{j-1}$, the binomial ${n+1\choose j}$ denotes the number of E-eigenvalues for $f$ whose corresponding E-eigenvectors have exactly $j$ non-zero coordinates, while the factor $(d-2)^{j-1}$ corresponds to the number of $(j-1)$-arrangements (allowing repetitions) of the elements of $\{0,1,\ldots,d-3\}$, for all $j=1,\ldots,n+1$. Let $x=(x_1,\ldots,x_{n+1})$, $\|x\|=1$ be an E-eigenvector of $f$. We have
\begin{equation}\label{fermat1}
a_ix_i^{d-1}=\lambda x_i\quad\forall i=1,\ldots,n+1.
\end{equation}
Suppose that exactly $j$ coordinates of $x$ are non-zero, call them $x_{k_1},\ldots,x_{k_j}$ with indices $1\leq k_1<\cdots<k_j\leq n+1$. Moreover, we write $a_i=\xi_i^{d-2}$ for all $i=1,\ldots,n+1$. Looking at (\ref{fermat1}), if $x_i\neq 0$ we obtain that $\lambda=a_ix_i^{d-2}=(\xi_ix_i)^{d-2}$ for all $i=1,\ldots,n+1$. Moreover, considering (\ref{fermat1}) with respect to the indices $i_1<i_2$, we get the relations $a_{i_1}x_{i_1}^{d-1}=\lambda x_{i_1}$, $a_{i_2}x_{i_2}^{d-1}=\lambda x_{i_2}$, from which we obtain the equation $x_{i_1}x_{i_2}\prod_{k=0}^{d-3}(\xi_{i_1}x_{i_1}-\varepsilon^k\xi_{i_2}x_{i_2})=0$, where $\varepsilon$ is a $(d-2)$-th root of unity. This means that, for any indices $i_1<i_2$ it could be that $x_{i_1}=0$, $x_{i_2}=0$ or $\xi_{i_1}x_{i_1}=\varepsilon^k\xi_{i_2}x_{i_2}$ for some $k\in\{0,\ldots,d-3\}$. Therefore the coordinates of $x$, when non-zero, can be always written as $x_{k_l}=\xi_{k_1}\cdots\widehat{\xi_{k_l}}\cdots\xi_{k_j}\varepsilon^{\alpha_{k_l}}/\|x\|$, where $\alpha_{k_l}\in\{0,1,\ldots,d-3\}$ for all $l=1,\ldots,j$. Since $|\varepsilon|=1$, we can assume $\alpha_{k_1}=0$. In addition to this, the norm of $x$ can be written as $\|x\|=\left(\xi_{k_2}^2\cdots\xi_{k_j}^2+\sum_{l=2}^j\xi_{k_1}^2\cdots\widehat{\xi_{k_l}^2}\cdots\xi_{k_j}^2\varepsilon^{2\alpha_{k_l}}\right)^{1/2}$
and the E-eigenvalue corresponding to $x$ is $\lambda=(\xi_{k_l}x_{k_l})^{d-2}=a_{k_1}\cdots a_{k_j}/\|x\|^{d-2}$. From this argument we obtain that the product of the E-eigenvalues of the scaled Fermat polynomial $f$ is equal to $\lambda_1\cdots\lambda_N = g/h^{(d-2)/2}$, where
\begin{align}\label{big product}
g=g(a_1,\ldots,a_{n+1}) & \coloneqq \prod_{j=1}^{n+1}\prod_{1\leq k_1<\cdots<k_j\leq n+1}\prod_{\alpha_{k_2},\ldots,\alpha_{k_j}=0}^{d-3}a_{k_1}\cdots a_{k_j},\\
h=h(a_1,\ldots,a_{n+1}) & \coloneqq \prod_{j=1}^{n+1}\prod_{1\leq k_1<\cdots<k_j\leq n+1}\prod_{\alpha_{k_2},\ldots,\alpha_{k_j}=0}^{d-3}\left(\xi_{k_2}^2\cdots\xi_{k_j}^2+\sum_{l=2}^j\xi_{k_1}^2\cdots\widehat{\xi_{k_l}^2}\cdots\xi_{k_j}^2\varepsilon^{2\alpha_{k_l}}\right).
\end{align}
Now consider in particular the polynomial $g$ defined in (\ref{big product}). We have that
\[
g = \prod_{j=1}^{n+1}\prod_{1\leq k_1<\cdots<k_j\leq n+1}(a_{k_1}\cdots a_{k_j})^{(d-2)^{j-1}} = \prod_{j=1}^{n+1}(a_1\cdots a_{n+1})^{{n\choose j-1}(d-2)^{j-1}} = (a_1\cdots a_{n+1})^{(d-1)^n},
\]
where the last polynomial coincides exactly with $\operatorname{Res}\left(\frac{1}{d}\nabla f\right)$ by Proposition \ref{def resultant}. On the other hand, having fixed $Lex$ as term order in $\mathbb{Z}[a_1,\ldots,a_{n+1}]$, the leading term of $h$ is equal to
\[
LT_{Lex}(h) = \prod_{j=2}^{n+1}\prod_{1\leq k_1<\cdots<k_j\leq n+1}\prod_{\alpha_{k_2},\ldots,\alpha_{k_j}=0}^{d-3}\xi_{k_1}^2\cdots\xi_{k_{j-1}}^2\varepsilon^{2\alpha_{k_j}} = \prod_{j=2}^{n+1}\prod_{1\leq k_1<\cdots<k_j\leq n+1}(\xi_{k_1}^2\cdots\xi_{k_{j-1}}^2)^{(d-2)^{j-1}}.
\]
% &= \prod_{s=1}^{n+1}\xi_{s}^{2\sum_{j=2}^{n+1}\left[{n\choose j-1}-{s-1\choose j-1}\right](d-2)^{j-1}}\\
% &= \prod_{s=1}^{n+1}\xi_{s}^{2\left[(d-1)^n-1-\sum_{j=2}^{n+1}{s-1\choose j-1}(d-2)^{j-1}\right]}\\
% &= \prod_{s=1}^{n+1}\xi_{s}^{2\left[(d-1)^n-1-\sum_{j=2}^{s}{s-1\choose j-1}(d-2)^{j-1}\right]}\\
Observe that in the last product (with $j$ fixed) the factors $\xi_1^{2(d-2)^{j-1}},\ldots,\xi_{j-1}^{2(d-2)^{j-1}}$ appear ${n\choose j-1}$ times, while $\xi_s^{2(d-2)^{j-1}}$ appears ${n\choose j-1}-{s-1\choose j-1}$ times for $s=j,\ldots,n+1$. Hence (assuming ${a\choose b}=0$ for $a<b$)
\[
LT_{Lex}(h) = \prod_{j=2}^{n+1}\prod_{s=1}^{n+1}\xi_{s}^{2\left[{n\choose j-1}-{s-1\choose j-1}\right](d-2)^{j-1}} = \prod_{s=1}^{n+1}\xi_{s}^{2\left[(d-1)^n-(d-1)^{s-1}\right]} = \prod_{s=1}^{n+1}a_{s}^{2\frac{(d-1)^n-(d-1)^{s-1}}{d-2}}.\qedhere
\]
\end{proof}

\begin{rmk}
Observe that in Lemma \ref{scaled Fermat polynomials} the degree of $LT_{Lex}(h)$, that is the total degree of $h$, is
\[
\frac{2}{d-2}\sum_{s=1}^{n+1}[(d-1)^n-(d-1)^{s-1}]=\frac{2}{d-2}\left((n+1)(d-1)^n-\frac{(d-1)^{n+1}-1}{d-2}\right)=\frac{2}{d-2}\varphi_n(d),
\]
where $\varphi_n(d)$ has been introduced in Proposition \ref{degrees of coefficients of charpoly} and $d\geq 3$. This value is the one expected as showed in the sequel.
\end{rmk}

Now we move to the general case. We recall the definition of the \emph{polar classes} $\delta_j(X)$ associated to a projective variety $X\subseteq\mathbb{P}^n$ of dimension $r$.

Consider the conormal variety $Z(X)$ of $X$ introduced in Definition \ref{conormal variety} and the \emph{Chow cohomology class}
\[
[Z(X)]\in A^*(\mathbb{P}^n\times(\mathbb{P}^n)^{\vee})=\mathbb{Z}[u,v],
\]
where $A^*(\mathbb{P}^n\times(\mathbb{P}^n)^{\vee})$ denotes the Chow (cohomology) ring of $\mathbb{P}^n\times(\mathbb{P}^n)^{\vee}$, $u=pr^*_1([H])$, $v=pr^*_2([H'])$ and $H$, $H'$ denote hyperplanes in $\mathbb{P}^n$ and $(\mathbb{P}^n)^{\vee}$, respectively. Then $[Z(X)]$ can be written as
\[
[Z(X)]=\sum_{j=0}^{r-1}\delta_j(X)u^{n-j}v^{j+1},
\]
where $\delta_j(X)$ is a non-negative integer for all $j=0,\ldots,r-1$. If $X$ is smooth, we have the following formulas for the invariants $\delta_j(X)$ (see \cite{holme1988geometric}):
\begin{equation}\label{formulas delta_j(X)}
\delta_j(X)=\sum_{k=j}^{r}(-1)^{r-k}{k+1\choose j+1}\deg(c_{r-k}(TX)),
\end{equation}
where $\deg(c_{r-k}(TX))$ is the degree of the $(r-k)$-th Chern class of the tangent bundle of $X$.

We will use the following result (see \cite[Theorem 3.4]{holme1988geometric}).
\begin{thm}\label{dimension dual of X}
If $\delta_j(X)=0$ for all $j<l$ and $\delta_l(X)\neq 0$, then $\dim(X^{\vee})=n-1-l$ and $\delta_l(X)=\deg(X^{\vee})$.
\end{thm}

We apply the previous facts in our particular case. Let $v_{n,d}$ be the Veronese embedding defined in the introduction, and denote by $\widetilde{Q}$ the isotropic quadric $Q$ embedded in $\mathbb{P}^N$ via the map $v_{n,d}$, where $N+1={n+d\choose d}$. In particular, $\widetilde{Q}$ is smooth, hence we can apply the relations (\ref{formulas delta_j(X)}). Moreover, it is known that
\[
\widetilde{Q}^{\vee}=\overline{\left\{[f]\in(\mathbb{P}^N)^{\vee}\ |\ [f]\ \mbox{is tangent to $Q$ at some smooth point}\right\}}.
\]

\begin{lmm}\label{computation of invariant delta_0}
In the hypotheses above, $\delta_0(\widetilde{Q}) = 2\sum_{k=0}^{n-1}\alpha_kd^k$, where
\begin{equation}\label{degree isotropic quadric embedded via Veronese map}
\alpha_k \coloneqq (k+1)\sum_{j=0}^{n-1-k}{n+1\choose j}(-1)^j2^{n-1-k-j}.
\end{equation}
\end{lmm}
\begin{proof}
First of all we compute the Chern polynomial of $TQ$:
\[
c(TQ) = \frac{(1+t)^{n+1}}{1+2t} = \sum_{i,j=0}^{n-1}{n+1\choose i}(-2)^jt^{i+j} = \sum_{s=0}^{n-1}\left(\sum_{i=0}^s{n+1\choose i}(-2)^{s-i}\right)t^s.
\]
Then we compute the polar class $\delta_0(\widetilde{Q})$ using (\ref{cond1}) with $r=n-1$ and taking into account that~$\widetilde{Q}~=~v_{n,d}(Q)$.
\begin{align*}
\delta_0(\widetilde{Q}) & = \sum_{k=0}^{n-1}(-1)^{n-1-k}(k+1)\deg(c_{n-1-k}(T\widetilde{Q}))\\
 & = \sum_{k=0}^{n-1}(-1)^{n-1-k}(k+1)\deg\left[\left(\sum_{j=0}^{n-1-k}{n+1\choose j}(-2)^{n-1-k-j}\right)t^{n-1-k}(dt)^{k}\right]\\
 & = \sum_{k=0}^{n-1}(k+1)\deg\left[\left(\sum_{j=0}^{n-1-k}{n+1\choose j}(-1)^j2^{n-1-k-j}\right)d^kt^{n-1}\right]\\
 & = 2\sum_{k=0}^{n-1}(k+1)\left[\sum_{j=0}^{n-1-k}{n+1\choose j}(-1)^j2^{n-1-k-j}\right]d^k.\qedhere
\end{align*}
%Setting $\alpha_k$ as in (\ref{degree isotropic quadric embedded via Veronese map}), we conclude the proof.
\end{proof}

In the following technical Lemma, we rewrite the polynomial $\varphi_n(d)$ defined in Proposition \ref{degrees of coefficients of charpoly} in a useful way for the sequel.

\begin{lmm}\label{computation alpha_k}
Let $\varphi_n(d)$ be the polynomial defined in Proposition \ref{degrees of coefficients of charpoly}. Then $\varphi_n(d)=(d-2)\sum_{k=0}^{n-1}\beta_{k}d^k$, where
\begin{equation}\label{eq: computation beta_k}
\beta_{k}\coloneqq(k+1)\sum_{l=0}^{n-1-k}{k+l+1\choose l}(-1)^{l}.
\end{equation}
\end{lmm}
\begin{proof}
With a bit of work, the polynomial $\varphi_n(d)$ can be rewritten as
\[
\varphi_n(d) = (d-2)\sum_{k=0}^{n-1}(k+1)(d-1)^k = (d-2)\sum_{k=0}^{n-1}\beta_{k}d^k,
\]
where
\begin{align*}
\beta_{k} & = \sum_{i=k}^{n-1}(i+1){i\choose k}(-1)^{i-k}\\
 & = \sum_{l=0}^{n-1-k}(l+k+1){l+k\choose k}(-1)^{l}\\
% & = \sum_{l=0}^{n-1-k}(l+k+1)\frac{(l+k)!}{l!k!}(-1)^{l}\\
 & = \sum_{l=0}^{n-1-k}\left[\frac{l(l+k)!}{l!k!}+\frac{(k+1)(l+k)!}{l!k!}\right](-1)^{l}\\
% & = \sum_{l=0}^{n-1-k}\left[\frac{(k+1)(l+k)!}{(l-1)!(k+1)!}+\frac{(k+1)(l+k)!}{l!k!}\right](-1)^{l}\\
 & = (k+1)\sum_{l=0}^{n-1-k}\left[{k+l\choose k+1}+{k+l\choose k}\right](-1)^{l}\\
 & = (k+1)\sum_{l=0}^{n-1-k}{k+l+1\choose l}(-1)^{l}.\qedhere
\end{align*}
\end{proof}

Now we prove that the degree of the leading coefficient of $\psi_f$ is a multiple of the polar class $\delta_0(\widetilde{Q})$ computed in Lemma \ref{computation of invariant delta_0}.

\begin{prop}\label{identity alpha beta}
For any $n\geq 1$ and for any $k=0,\ldots,n-1$, $\alpha_{k}=\beta_{k}$. In particular,
\begin{equation}\label{eq: identity alpha beta}
\varphi_n(d)=\frac{d-2}{2}\delta_0(\widetilde{Q}).
\end{equation}
\end{prop}
\begin{proof}
From the identities (\ref{degree isotropic quadric embedded via Veronese map}) and (\ref{eq: computation beta_k}) we see that both $\alpha_{k}$ and $\beta_{k}$ are multiples of $k+1$. In particular, we have to prove that
\begin{equation}\label{equality alpha_k=beta_k}
\sum_{j=0}^{n-1-k}{n+1\choose j}(-1)^j2^{n-1-k-j}=\sum_{j=0}^{n-1-k}{k+j+1\choose j}(-1)^{j}.
\end{equation}
The proof is by induction on $n$. If $n=1$, both the sides of the equality are equal to 1. Suppose now that the equality is true at the $n$-th step. At the $(n+1)$-th step, the right-hand side of the equality is
\[
\sum_{j=0}^{n-k}{k+j+1\choose j}(-1)^{j}=\sum_{j=0}^{n-1-k}{k+j+1\choose j}(-1)^j+(-1)^{n-k}{n+1\choose n-k},
\]
while the left-hand side at the $(n+1)$-th step is equal to
\begin{align*}
\sum_{j=0}^{n-k}{n+2\choose j}(-1)^j2^{n-k-j} & = 2^{n-k}+\sum_{j=1}^{n-k}{n+2\choose j}(-1)^j2^{n-k-j}\\
 & = 2^{n-k}+\sum_{j=1}^{n-k}\left[{n+1\choose j}+{n+1\choose j-1}\right](-1)^j2^{n-k-j}\\
 & = \sum_{j=0}^{n-k}{n+1\choose j}(-1)^j2^{n-k-j}+\sum_{j=1}^{n-k}{n+1\choose j-1}(-1)^j2^{n-k-j}\\
 & = 2\sum_{j=0}^{n-1-k}{n+1\choose j}(-1)^j2^{n-1-k-j}+(-1)^{n-k}{n+1\choose n-k}+\\
 & \ \ \ \ +\sum_{j=0}^{n-1-k}{n+1\choose j}(-1)^{j+1}2^{n-1-k-j}\\
 & = \sum_{j=0}^{n-1-k}{n+1\choose j}(-1)^j2^{n-1-k-j}+(-1)^{n-k}{n+1\choose n-k}
\end{align*}
By inductive case we prove (\ref{equality alpha_k=beta_k}).
\qedhere
\end{proof}

\begin{rmk}\label{rmk: Matteo Gallet}
Matteo Gallet suggested an alternative proof of the identity (\ref{equality alpha_k=beta_k}), applying the so-called \emph{``Zeilberger's Algorithm''} (see \cite{zeilberger1990holonomic,zeilberger1991telescoping}). For example, using the Mathematica package \verb+HolonomicFunctions+, developed by Cristoph Koutschan (see \cite{koutshan2010telescoping}), the code
{\small
\begin{verbatim}
Annihilator[Sum[Binomial[n+1,j]*(-1)^j*2^(n-1-k-j),{j,0,n-1-k}],{S[k],S[n]}]
Annihilator[Sum[Binomial[k+j+1,j]*(-1)^j,{j,0,n-1-k}],{S[k],S[n]}]
\end{verbatim}
}
\noindent provides the operators that annihilate the left-hand and right-hand side in (\ref{equality alpha_k=beta_k}), respectively, thus showing that (\ref{equality alpha_k=beta_k}) holds true.
\end{rmk}

\begin{cor}\label{cor: Q tilde is a hypersurface}
Consider the isotropic quadric $Q\subseteq\mathbb{P}^n$ and its Veronese embedding $\widetilde{Q}\subseteq\mathbb{P}^N$ with the same notations as before. Then $\widetilde{Q}^{\vee}$ is a hypersurface of $(\mathbb{P}^n)^{\vee}$ of degree $\deg(\widetilde{Q}^{\vee})=\delta_0(\widetilde{Q})$.
\end{cor}
\begin{proof}
From Remark \ref{rmk: on the polynomial phi(n,d)} and Proposition \ref{identity alpha beta} we have that $\delta_0(\widetilde{Q})$ is a positive integer for all $n\geq 1$ and $d>2$. Applying Theorem \ref{dimension dual of X} we conclude the proof.
\end{proof}

Summing up, there is an explicit formula for the degree of the leading coefficient of $\psi_f$ in terms of the degree of the dual variety of $Q$ embedded in $\mathbb{P}^N$ via the Veronese map, stated in the following corollary.

\begin{cor}\label{exponent power}
Given $f\in\operatorname{Sym}^dV$, if $f$ is general then
\[
deg(c_{N})=\frac{d-2}{2}\deg(\widetilde{Q}^{\vee})
\]
when $d$ is even, while
\[
deg(c_{2N})=(d-2)\deg(\widetilde{Q}^{\vee})
\]
when $d$ is odd.
\end{cor}
In the following, we prove that the leading coefficient of $\psi_f$ is a power of the discriminant $\Delta_{\widetilde{Q}}(f)$, where the exponent has been obtained in Corollary \ref{exponent power}. The next two lemmas will clarify the geometrical meaning of the vanishing of the polynomial $c_{N}$ (respectively $c_{2N}$).

\begin{lmm}\label{lemma 7.1 cinesi}
Assume that $d>2$ and let $f\in\operatorname{Sym}^d V$. Then the leading coefficient of $\psi_f$ vanishes if and only if the system
\begin{equation}\label{deficit system}
\left\{ 
\begin{array}{l}
\frac{1}{d}\nabla f(x) = \lambda x\\
\|x\| = 0\ ,
\end{array}
\right.
\end{equation}
called \emph{deficit system} in \cite{li2012characteristic}, has a nontrivial solution.
\end{lmm}
\begin{proof}
If $f$ is irregular, from Definition \ref{irregular tensor} we have that the system (\ref{deficit system}) has a non trivial solution when $\lambda=0$, while from Proposition \ref{zero charpoly} we have that $\psi_f$ is identically zero.

Suppose instead that $f$ is regular. By Proposition \ref{eigenvalues as roots of charpoly} the roots of $\psi_f$ are exactly the E-eigenvalues of~$f$ and for $d$ even $\deg(\psi_f)\leq N$, whereas for $d$ odd $\deg(\psi_f)\leq 2N$. However, we know by Theorem \ref{expected number of E-eigenvalues} that a general $f$ has $N$ distinct E-eigenvalues when $d$ is even, and $N$ pairs $(\lambda,-\lambda)$ of distinct E-eigenvalues when $d$ is odd, which means that $\psi_f$ would have exactly $N$ distinct roots when $d$ is even, and $2N$ distinct roots when $d$ is odd. On the other hand, E-eigenvalues are the normalized solutions $x$ of equation (\ref{new interpretation eigenvectors}), and by definition $\psi_f$ is the resultant of the homogeneization of the system whose equations are (\ref{new interpretation eigenvectors}) and the condition $\|x\|=1$. The solutions at infinity of this system are precisely the solution of the system~(\ref{deficit system}). Hence a symmetric tensor $f$ such that $\psi_f$ has not the maximal degree provides a nontrivial solution of the system (\ref{deficit system}), or equivalently admits an isotropic eigenvector.  
\end{proof}
\begin{lmm}\label{lemma 7.2 cinesi}
Given $f\in\operatorname{Sym}^dV$, the system (\ref{deficit system}) has a nontrivial solution if and only if the coefficients of $f$ annihilate the polynomial $\Delta_{\widetilde{Q}}(f)$, namely $f$ is represented by a point of $\widetilde{Q}^{\vee}$.
\end{lmm}
\begin{proof}
Suppose that $x$ is a solution of (\ref{deficit system}). By regularity of $f$ we have that $\lambda\neq 0$. Moreover, $P=[x]$ is a smooth point of $f$, and $f$ is tangent to $Q$ at $P$. This means that $f$, thought as a point of $\mathbb{P}(\operatorname{Sym}^dV)^{\vee}$, belongs to $\widetilde{Q}^{\vee}$, namely its coefficients annihilate the polynomial $\Delta_{\widetilde{Q}}(f)$. The converse is true by reversing the implications.\qedhere
\end{proof}

\begin{rmk}\label{rmk: on transversality}
One could ask if the condition on $f$ to have the maximum number of E-eigenvalues imposed in the Main Theorem has a geometric counterpart. For example, this condition is not the same as requiring $[f]$ to be regular: although any symmetric tensor $f$ having the maximum number of E-eigenvalues is necessarily regular, there exist regular symmetric tensors $f$ admitting at least one isotropic eigenvector. The right property to consider is revealed by Lemma \ref{lemma 7.2 cinesi}, which shows that $f\in\operatorname{Sym}^d V$ admits an isotropic eigenvector if and only if the hypersurface $[f]$ and the isotropic quadric $Q$ are tangent. This means that the condition on $f$ in the Main Theorem is satisfied if and only if $[f]$ is transversal to $Q$.
\end{rmk}

Remark \ref{rmk: on transversality} is even more interesting when considering the following result (see \cite[Claim 3.2]{aluffi00}):

\begin{prop}\label{prop: aluffi}
If two smooth hypersurfaces of degree $d_1$, $d_2$ in projective space are tangent along a positive dimensional set, then $d_1=d_2$.
\end{prop}

An immediate consequence of Proposition \ref{prop: aluffi} is the following

\begin{cor}\label{cor: smooth symmetric tensor}
Given $f\in\operatorname{Sym}^d V$ with $d>2$, if $[f]$ is smooth then $f$ has always a finite number of isotropic eigenvectors.
\end{cor}

A detailed example of a symmetric tensor $f$ admitting an isotropic eigenvector, with a study of the tangency of the variety defined by $f$  with the isotropic quadric $Q$, is given in Section \ref{ex: plane cubic isotropic eigenvector}.

Returning to the proof of the Main Theorem, an immediate consequence of Corollary \ref{exponent power} and Lemmas \ref{lemma 7.1 cinesi} and \ref{lemma 7.2 cinesi} is the following formula for the leading coefficient of the E-characteristic polynomial of a symmetric tensor.

\begin{thm}\label{thm: leading coefficient}
Given $f\in\operatorname{Sym}^dV$ and $d>2$, if $f$ does not admit isotropic eigenvectors, then
\begin{equation}\label{eq: leading coefficient, even case}
c_{N}=e\cdot\Delta_{\widetilde{Q}}(f)^{\frac{d-2}{2}}
\end{equation}
when $d$ is even, while
\begin{equation}\label{eq: leading coefficient, odd case}
c_{2N}=e\cdot\Delta_{\widetilde{Q}}(f)^{d-2}
\end{equation}
when $d$ is odd, for some integer constant $e=e(n,d)$.
\end{thm}
\begin{proof} Applying Lemma \ref{lemma 7.1 cinesi} and Lemma \ref{lemma 7.2 cinesi} we obtain that the varieties $\{c_{N}=0\}$ and $\{\Delta_{\widetilde{Q}}(f)=0\}$ coincide. The proof for the case $n=1$ is postponed to Section \ref{ex: binary forms}, where we treat more in detail binary forms. If $n>1$, then $\widetilde{Q}$ is an irreducible hypersurface and the variety $\widetilde{Q}^{\vee}$ is irreducible as well. Corollary \ref{cor: Q tilde is a hypersurface} tells us that $\widetilde{Q}^{\vee}$ is in fact a hypersurface. Hence, for $d$ even, $c_{N}=e\cdot\Delta_{\widetilde{Q}}(f)^j$, whereas for $d$ odd $c_{2N}=e\cdot\Delta_{\widetilde{Q}}(f)^k$ for some integer constant $e=e(n,d)$ and positive integers $j,k$. Moreover, from Corollary \ref{exponent power} we have that $j=(d-2)/2$ and $k=d-2$.\qedhere

%In order to determine the constants $\zeta_{n,d}$ explicitly, we consider as example the class of scaled Fermat polynomials introduced in Lemma \ref{scaled Fermat polynomials}. We know that the product of the E-eigenvalues is equal (up to sign as explained in the introduction) to $\frac{c_0(n,d)}{c_{h(n,d)}(n,d)}$ for $d$ even, and is equal to $\frac{c_0(n,d)}{c_{2h(n,d)}(n,d)}$ for $d$ odd. Furthermore, from Lemma \ref{scaled Fermat polynomials} the numerator in (\ref{product_eigenvalues}) is equal to $c_0(n,d)$ when $d$ is even, and is equal to $c_0(n,d)^{\frac{1}{2}}$ when $d$ is odd. This means that the denominator of equation (\ref{product_eigenvalues}) coincides with the leading coefficient of $\psi_f$. By our definition, $\Delta_{\widetilde{Q}}(f)$ is irreducible over $\mathbb{Z}$. This fact, together with the fact that the leading term of the denominator in (\ref{product_eigenvalues}) is monic, tells us that $\zeta_{n,d}=1$ for all $n$ and $d$. This means that the product of the E-eigenvalues of a symmetric tensor has the form given in equation (\ref{product eigenvalues of the theorem}).
\end{proof}

\begin{proof}[Proof of the Main Theorem]
Theorems \ref{thm: constant term of charpoly} and \ref{thm: leading coefficient} describe respectively the constant term $c_0$ and the leading coefficient $c_N$ (or $c_{2N}$) of the E-characteristic polynomial $\psi_f$ of a generic symmetric tensor $f$, up to a constant integer factor. Moreover, the product of the E-eigenvalues of $f$ is $c_0/c_N$ (respectively $c_0/c_{2N}$). If we restrict to the class of scaled Fermat polynomials, as in Lemma \ref{scaled Fermat polynomials}, we notice that the integers $c$ and $e$ of Theorems \ref{thm: constant term of charpoly} and \ref{thm: leading coefficient} have to coincide, for the leading term of the denominator in (\ref{product_eigenvalues}) is monic and by definition $\Delta_{\widetilde{Q}}(f)$ has relatively prime integer coefficients. This concludes the proof.\qedhere
\end{proof}
\section{Examples with binary and ternary symmetric tensors} \label{section: examples}

In this section we give two examples to understand better the statement and the proof of the Main Theorem. The first one deals with the case of binary forms: in particular, we show that in this particular case equation (\ref{product eigenvalues of the theorem}) can be rewritten more explicitly. The second is an example of a cubic ternary form $f$ which admits one isotropic eigenvector: we compute explicitly its E-characteristic polynomial $\psi_f$, observe that $\deg(\psi_f)<N$ and visualize its tangency with the isotropic quadric $Q$.

\subsection{The case of binary forms}\label{ex: binary forms}

In this example we focus on the case $n=1$ and recover the results of Li, Qi and Zhang in \cite{li2012characteristic}. An element of $\operatorname{Sym}^d\mathbb{C}^2$ is represented by the binary form
\begin{equation}\label{eq: definition of f}
f(x_1,x_2)=\sum_{j=0}^d\binom{d}{j}a_jx_1^{d-j}x_2^j\ ,\quad a_0,\ldots,a_d\in\mathbb{C}.
\end{equation}
According to Theorem \ref{expected number of E-eigenvalues}, a general binary form $f$ of degree $d$ admits $N=d$ E-eigenvectors. As one can easily see from relation (\ref{new interpretation eigenvectors}), the E-eigenvectors of $f$ are the normalized solutions $(x_1,x_2)$ of the equation $D(f)=0$, where the \emph{discriminant operator} $D$ is defined by $D(f)\coloneqq x_1(\partial f/\partial x_2)-x_2(\partial f/\partial x_1)$. The operator $D$ is well-known and its properties are collected in \cite{maccioni1972number}.

We are interested in the E-characteristic polynomial $\psi_f$ of a regular binary form $f$. We know that $\deg(\psi_f)=d$ in the even case, while $\deg(\psi_f)=2d$ in the odd case. A remarkable formula for the leading coefficient of the E-characteristic polynomial of a $2$-dimensional tensor of order $d$ is given in \cite{li2012characteristic}. We show that this formula can be simplified a lot in the symmetric case.

Following the argument used in \cite{li2012characteristic}, the isotropic eigenvectors of $f$ are the solutions of the following simplified version of the system (\ref{deficit system}):
\begin{equation}\label{simplified deficit system}
\left\{ 
\begin{array}{lll}
\sum_{j=1}^d{d-1\choose j-1}a_{j-1}x_1^{d-j}x_2^{j-1} & = & \lambda x_1\\
\sum_{j=1}^d{d-1\choose j-1}a_{j}x_1^{d-j}x_2^{j-1} & = & \lambda x_2\\
x_1^2+x_2^2 & = & 0\ .
\end{array}
\right.
\end{equation}
We observe that all the non trivial solutions $(x_1,x_2)$ of (\ref{simplified deficit system}) are non-zero multiples of $(1,\sqrt{-1})$ or $(1,-\sqrt{-1})$. Substituting $(1,\sqrt{-1})$ to (\ref{simplified deficit system}) and eliminating $\lambda$ we obtain the condition
\begin{equation}\label{cond1}
\sum_{j=0}^d{d\choose j}a_{j}\sqrt{-1}^{j}=0.
\end{equation}
In the same manner, considering instead the vector $(1,-\sqrt{-1})$ we obtain the condition
\begin{equation}\label{cond2}
\sum_{j=0}^d{d\choose j}a_{j}(-\sqrt{-1})^{j}=0. 
\end{equation}
Therefore, if the binary form $f$ has at least one isotropic eigenvector, then the product of the left-hand sides of equations (\ref{cond1}) and (\ref{cond2}) vanishes. On the other hand, if this product is zero, then $(1,\sqrt{-1})$ or $(1,-\sqrt{-1})$ is a solution of the system (\ref{simplified deficit system}) and is in turn an isotropic eigenvector of $f$.

We observe that the left-hand sides in (\ref{cond1}) and (\ref{cond2}) have an interesting interpretation. Consider in general the linear change of coordinates defined by the equations
\[
x_1=\gamma_{11}z_1+\gamma_{12}z_2,\quad x_2=\gamma_{21}z_1+\gamma_{22}z_2.
\]
Applying this change of coordinates, the binary form $f(x_1,x_2)$ is transformed into the binary form $\widetilde{f}(z_1,z_2)$ in the new variables $z_1$, $z_2$ defined by
\[
\widetilde{f}(z_1,z_2)=\sum_{j=0}^d\binom{d}{j}a_j(\gamma_{11}z_1+\gamma_{12}z_2)^{d-j}(\gamma_{21}z_1+\gamma_{22}z_2)^j=\sum_{j=0}^d\binom{d}{j}\widetilde{a}_jz_1^{d-j}z_2^j,
\]
where (see \cite[Proposition 3.6.1]{sturmfels2008algorithms})
\begin{equation}\label{b_i's}
\widetilde{a}_j=\sum_{k=0}^d\left[\sum_{l=\max(0,k-j)}^{\min(k,d-j)}\binom{d-j}{l}\binom{j}{k-l}\gamma_{11}^l\gamma_{12}^{k-l}\gamma_{21}^{d-j-l}\gamma_{22}^{j-k+l}\right]a_k,\quad j=0,\ldots,d.
\end{equation}
In particular consider the new coordinates
\[
z_1=-\frac{\sqrt{-1}}{2}(x_1+\sqrt{-1}\hspace{0.5mm}x_2),\quad z_2=-\frac{\sqrt{-1}}{2}(x_1-\sqrt{-1}\hspace{0.5mm}x_2).
\]
The inverse change of coordinates has equations
\[
x_1=\sqrt{-1}(z_1+z_2),\quad x_2=z_1-z_2.
\]
With this choice, applying formula (\ref{b_i's}) the coefficients $\widetilde{a}_j$ of the transformed binary form $\widetilde{f}(z_1,z_2)$ are
\[
\widetilde{a}_j=\sum_{k=0}^d\left[\sum_{l=\max(0,k-j)}^{\min(k,d-j)}\binom{d-j}{l}\binom{j}{k-l}\sqrt{-1}^{2(j+l)-k}\right]a_k,\quad j=0,\ldots,d.
\]
In particular the extreme coefficients become
\[
\widetilde{a}_0=\sum_{j=0}^d{d\choose j}a_{j}\sqrt{-1}^{j},\quad \widetilde{a}_d=(-1)^d\sum_{j=0}^d{d\choose j}a_{j}(-\sqrt{-1})^{j}.
\]
Therefore, if we define $b_0\coloneqq\widetilde{a}_0$ and $b_d\coloneqq(-1)^d\widetilde{a}_d$, then the left-hand sides of equations (\ref{cond1}) and (\ref{cond2}) are equal to $b_0$ and $b_d$, respectively. Moreover, we observe that the product $b_0b_d$ has integer coefficients even though some of the coefficients of $b_0$ and $b_d$ have non-zero imaginary part: in fact we see that
\begin{equation}\label{product b_0b_d}
b_0b_d = \sum_{j,k=0}^d{d\choose j}{d\choose k}a_ja_k(-1)^j\sqrt{-1}^{j+k} = \sum_{s=0}^d\left[\sum_{j=0}^s{d\choose j}{d\choose s-j}a_ja_{s-j}(-1)^j\right]\sqrt{-1}^s,
\end{equation}
where in the last relation all summands corresponding to odd indices $s$ vanish. Since the coefficient of $a_0$ in the expression of $b_0b_d$ is 1, we conclude that $b_0b_d=\Delta_{\widetilde{Q}}(f)$ up to sign. In particular $\widetilde{Q}^{\vee}=\{b_0b_d=0\}$: in fact in this case $\widetilde{Q}$ is the union of two distinct points (more precisely, the classes of the rank one symmetric tensors $(x_1+\sqrt{-1}x_2)^d$ and $(x_1-\sqrt{-1}x_2)^d$), while the variety $\widetilde{Q}^{\vee}$ is the quadric union of the hyperplanes $\{b_0=0\}$, $\{b_d=0\}$. In particular, the hyperplane $\{b_0=0\}$ parametrizes the binary forms having $(1,\sqrt{-1})$ as isotropic eigenvector, while $\{b_d=0\}$ parametrizes the binary forms having $(1,-\sqrt{-1})$ as isotropic eigenvector.

Regarding the leading coefficient of the E-characteristic polynomial $\psi_f$, the previous argument suggests that it must coincide with $c\cdot b_0^ib_d^j$ for some $c=c(d)\in\mathbb{Z}$. Since $\psi_f$ is a polynomial in the indeterminates $a_0,\ldots,a_d$ with integer coefficients, it follows that $i=j$. Hence, for $d$ even, $c_d=e\cdot\Delta_{\widetilde{Q}}(f)^p$, whereas for $d$ odd $c_{2d}=e\cdot\Delta_{\widetilde{Q}}(f)^q$ for some $e=e(n,d)\in\mathbb{Z}$ and positive integers $p,q$. From Corollary \ref{exponent power} we have that $p=(d-2)/2$ and $q=d-2$, thus completing the proof of Theorem \ref{thm: leading coefficient} in the case $n=1$.

\begin{rmk}
If we specialize to the class of scaled Fermat binary forms $f(x_1,x_2)=\alpha x_1^d+\beta x_2^d$, $\alpha,\beta\in\mathbb{C}$, from relation (\ref{product b_0b_d}) we confirm the statement of Lemma \ref{scaled Fermat polynomials} by observing that
\[
\Delta_{\widetilde{Q}}(f)=\alpha^2+(1+(-1)^d)\sqrt{-1}^d\alpha\beta+\beta^2.
\]
\end{rmk}

\subsection{A plane cubic admitting an isotropic eigenvector}\label{ex: plane cubic isotropic eigenvector}
The following example has the goal to explain better Lemma \ref{lemma 7.2 cinesi}. First of all, we recall that, due to Theorem \ref{expected number of E-eigenvalues}, a general ternary form has $N=d^2-d+1$ E-eigenvalues.
Consider the cubic ternary form
\begin{align*}
f(x_1,x_2,x_3) &= 342\sqrt{-1}\hspace{0.5mm}{x}_{1}^{3}-522\sqrt{-1}\hspace{0.5mm}{x}_{1} {x}_{2}^{2}-389\sqrt{-1}\hspace{0.5mm}{x}_{1}^{2} {x}_{3}+79\sqrt{-1}\hspace{0.5mm}{x}_{2}^{2} {x}_{3}-474\sqrt{-1}\hspace{0.5mm}{x}_{1} {x}_{3}^{2}+\\
 & \quad+95\sqrt{-1}\hspace{0.5mm}{x}_{3}^{3} -773 {x}_{1}^{2} {x}_{2}+191 {x}_{2}^{3}-48 {x}_{1} {x}_{2} {x}_{3}+175 {x}_{2} {x}_{3}^{2}.
\end{align*}
It can be easily verified that the vector $x=(0,1,-\sqrt{-1})$ is an isotropic eigenvector of $f$. In particular the projective curve $[f]$ is tangent to the isotropic quadric $Q$ at $[x]\in\mathbb{P}^2$, and the common tangent line has equation $x_2-\sqrt{-1}\hspace{0.5mm}x_3=0$.
In order to represent graphically this situation, we consider the change of coordinates
\[
z_1=-\sqrt{-1}\hspace{0.5mm}x_1,\quad z_2=x_2+\sqrt{-1}\hspace{0.5mm}x_3,\quad z_3=x_2-\sqrt{-1}\hspace{0.5mm}x_3. 
\]
In the $z_i$'s the quadric $Q$ (the red curve in the affine representation of Figure \ref{fig: cubic_isotropic_eigenvector}) has equation $z_1^2-z_2z_3=0$. The image of the isotropic eigenvector $x$ is $z=(0,2,0)$, while the image of the projective curve $[f]$ (the blue curve in Figure \ref{fig: cubic_isotropic_eigenvector}) is the projective curve of equation
\[
g(z_1,z_2,z_3)=342 {z}_{1}^{3}+581 {z}_{1}^{2} {z}_{2}+192 {z}_{1}^{2} {z}_{3}+498 {z}_{1} {z}_{2} {z}_{3}+139 {z}_{2}^{2} {z}_{3}+24 {z}_{1} {z}_{3}^{2}+48 {z}_{2} {z}_{3}^{2}+4 {z}_{3}^{3}.
\]
\begin{figure}[h!]
  \begin{center}
    \includegraphics[scale=0.085]{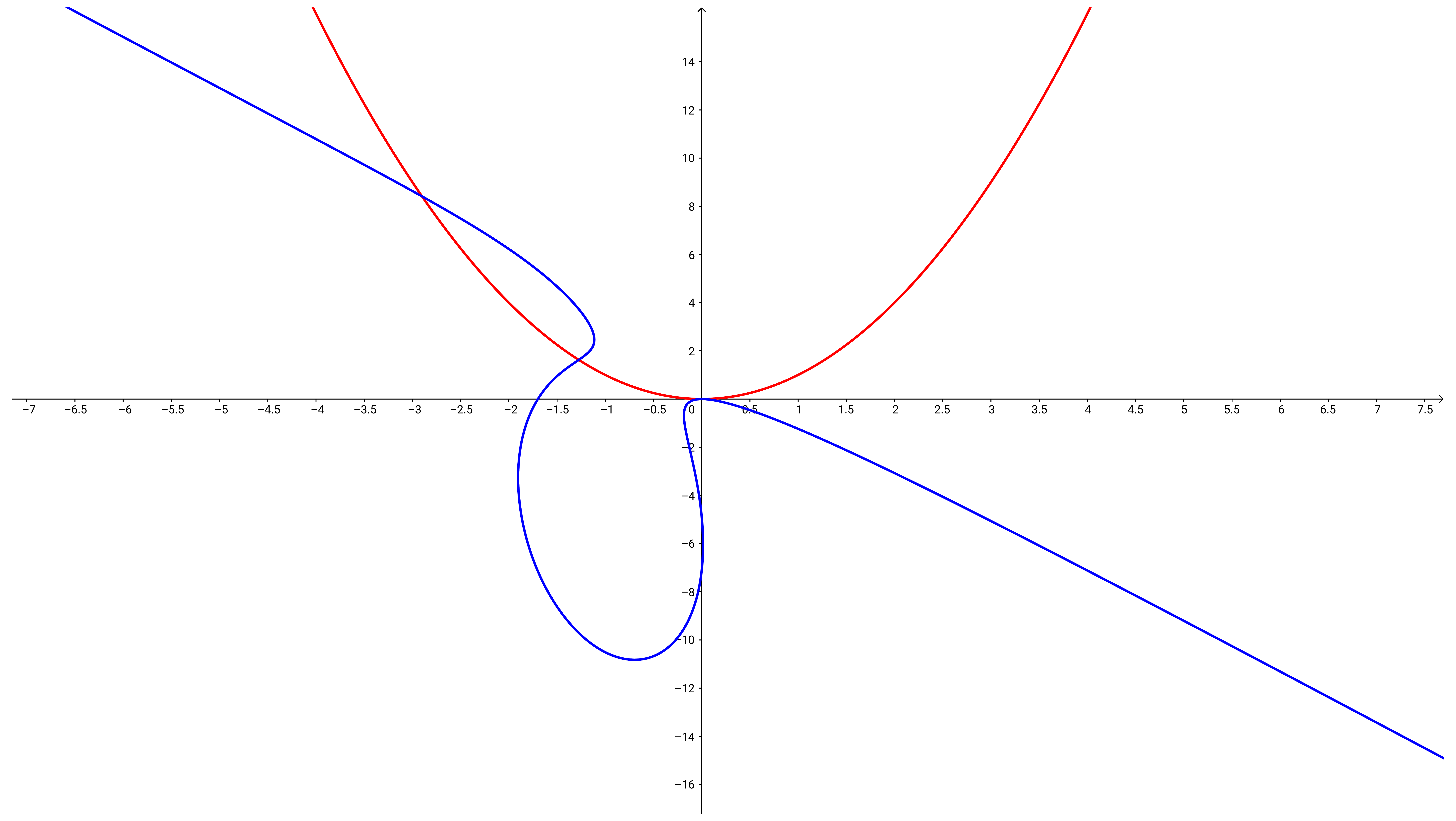}\\
    \caption{The isotropic quadric $Q$ and the ternary cubic $f$ in the affine plane $z_2=2$. They are not transversal at the origin.}
    \label{fig: cubic_isotropic_eigenvector}
  \end{center}
\end{figure}

\noindent The presence of an isotropic eigenvector can be detected by computing explicitly the E-characteristic polynomial of $f$ as well.
In order to compute $\psi_f(\lambda)$ we used the following Macaulay2 code \cite{graysonmacaulay} (for the package
\verb+Resultants+ see \cite{stagliano2017package}), taking into account Definition \ref{characteristic polynomial} modified according to the given change of coordinates:
{\footnotesize
\begin{verbatim}
loadPackage "Resultants"; KK=QQ[t]; R=KK[z_0..z_3];
f=342*z_1^3+581*z_1^2*z_2+192*z_1^2*z_3+498*z_1*z_2*z_3+139*z_2^2*z_3+24*z_1*z_3^2+48*z_2*z_3^2+4*z_3^3;
F_0=z_0^2-(-z_1^2+z_2*z_3); F_1=diff(z_1,f)/3+t*z_0*z_1;
F_2=diff(z_2,f)/3+diff(z_3,f)/3-t*z_0*(z_2+z_3)/2; F_3=diff(z_2,f)/3-diff(z_3,f)/3+t*z_0*(z_2-z_3)/2;
characteristic_polynomial=Resultant({F_0,F_1,F_2,F_3}, Algorithm=>Macaulay)
\end{verbatim}
}
\noindent The output of \verb+characteristic_polynomial+ is
{\footnotesize
\begin{align*}
\psi_g(\lambda) & =22405379203945800000 \lambda^{12}+1737672597491537284396875 \lambda^{10}+45686609440492531312122181875 \lambda^{8}\\
&\quad+538619871002221271247213134552625 \lambda^{6}+2746031584320556852962647720783548350 \lambda^{4}\\
&\quad+2137752598886514957981090279414043391031 \lambda^{2}+13843807659909379464027427753236120270069196.
\end{align*}
}
\noindent Since a general cubic ternary form has seven E-eigenvalues, we expect that $\deg(\psi_f)=14$, but in this case $\deg(\psi_f)=12$. This confirms that $f$ has one isotropic eigenvector and six E-eigenvectors (counted with multiplicity) up to sign.
\newpage
\section*{Acknowledgement}

Luca Sodomaco is member of INDAM-GNSAGA. This paper has been partially supported by the Strategic Project ``Azioni di gruppi su variet\`a e tensori'' of the University of Florence. The author is very grateful to his advisor Giorgio Ottaviani for valuable guidance. Moreover, he warmly thanks Matteo Gallet for Remark \ref{rmk: Matteo Gallet} and other appreciated suggestions.

\end{document}